\documentclass[reqno, 12pt]{amsart}
\usepackage{amsmath,amssymb,amsfonts,amsthm}
\usepackage{mathtools}
\usepackage{graphicx}
\usepackage{hyperref}
\usepackage{xcolor}
\usepackage{fullpage}
\usepackage{fancyhdr}
\fancypagestyle{firstpage}{%
\fancyhf{}
\fancyfoot[L]{\footnotesize 
\textbf{Corresponding author:} Shalmali Bandyopadhyay, Email: sbandyo5@utm.edu
}

}

\newtheorem{theorem}{Theorem}[section]
\newtheorem{rem}{Remark}[section]
\newtheorem{pro}{Proposition}[section]
\newtheorem{lemma}[theorem]{Lemma}
\newtheorem{definition}[theorem]{Definition}
\numberwithin{equation}{section}

\DeclareMathOperator{\rd}{rd}

\title{ Nonlinear Higher-Order Dynamic Equation with Polynomial Growth and Mixed Boundary Conditions }
\author{ Shalmali Bandyopadhyay$^\dagger$, \protect\footnote{$^\dagger$Department of Mathematics and Statistics, University of Tennessee at Martin, Tennessee, USA, Email: sbandyo5@utm.edu} Svetlin G. Georgiev$^*$ \protect\footnote{$^*$Department of Mathematics, Sorbonne University, Paris, France, Email: svetlingeorgiev1@gmail.com}}

\date{}

\begin{document}

\begin{abstract}
This paper investigates the existence of solutions for a class of nonlinear higher-order dynamic equations subject to mixed boundary conditions. We consider boundary value problems in which the nonlinear reaction functions satisfy polynomial growth conditions both in the interior of the domain and on the boundary. Our analysis employs a systematic approach based on fixed-point theorems for expansive mappings combined with completely continuous operators to establish stronger existence results. Under appropriate growth conditions on the nonlinear terms, we first prove the existence of at least one classical solution, which is not guaranteed to be nonnegative. We then strengthen our hypotheses to establish the existence of at least three nonnegative solutions. The theoretical framework relies on cone theory and carefully constructed open bounded subsets within function spaces equipped with appropriate norms. Our methodology provides a unified approach to multiple solution problems for higher-order dynamic equations.
\end{abstract}

\maketitle

\section{Introduction}


The study of boundary value problems for dynamic equations has attracted considerable attention since the foundational work of Hilger \cite{H}, who introduced the theory of time scales to unify differential and difference equations. The existence and multiplicity of solutions for such problems have been extensively investigated using various analytical techniques, particularly fixed-point theory and topological methods.

For second-order boundary value problems on time scales, significant progress has been made in establishing existence results for positive solutions. Anderson et al. \cite{A1,A2} developed criteria for the existence of positive solutions using cone theory and fixed-point theorems. Similarly, Henderson et al. \cite{H1,H2} studied multipoint boundary value problems, establishing existence and uniqueness results through solution matching techniques and various fixed-point arguments. The work on triple solutions to boundary value problems on time scales \cite{MT} demonstrated the effectiveness of cone-theoretic methods in obtaining multiple solution results.

The application of fixed-point theory to boundary value problems on time scales has evolved significantly, with researchers employing increasingly sophisticated techniques. Classical approaches have relied heavily on Schauder's fixed-point theorem and its variants for completely continuous operators \cite{S}. More recent developments have incorporated Krasnoselskii's fixed-point theorem \cite{K}, which combines the advantages of both contraction mapping principles and Schauder's theorem by considering the sum of contractive and compact operators.

Recent advances in fixed-point theory have focused on expansive mappings, which represent a natural counterpart to contractive mappings. Wang, Li, Gao, and Iseki \cite{WLGI} established fundamental results for expansion mappings in metric spaces, while subsequent work by Rhoades \cite{R}, Taniguchi \cite{T}, and others extended these results to pairs of mappings. More recently, Ahmed \cite{AH} developed common fixed-point theorems for expansive mappings in 2-metric spaces using compatibility conditions, and this work has been further generalized to G-metric spaces by removing restrictive continuity requirements \cite{SPB}.

The combination of expansive and completely continuous operators has emerged as a powerful approach in fixed-point theory. Recent work by Djebali and Mebarki \cite{3} on fixed-point index theory for perturbations of expansive mappings by k-set contractions has provided new tools for establishing existence results. Similarly, developments in cone theory and the use of carefully constructed open bounded subsets have enabled researchers to prove the existence of multiple solutions for various classes of boundary value problems \cite{LP,GP}.

For higher-order differential equations, the literature is more limited. Most existing results focus on specific cases such as third-order or fourth-order equations with particular boundary conditions. Zhou and Ma \cite{ZM} studied third-order boundary value problems with p-Laplacian operators, while Anderson \cite{A3} examined higher-order multipoint problems using monotone iterative techniques. However, a systematic treatment of general higher-order dynamic equations with polynomial growth conditions and mixed boundary conditions remains largely unexplored.

The existing literature reveals several gaps that our work addresses. First, while there are numerous results for second-order boundary value problems on time scales, the theory for higher-order problems is less developed, particularly for equations with polynomial growth conditions both in the interior and on the boundary. Second, although fixed-point theorems for expansive mappings and completely continuous operators have been studied separately, their systematic combination for establishing existence results (from one to multiple solutions) has not been fully exploited. Finally, most existing work provides either single solution existence or specific multiple solution results, but a unified approach that demonstrates the systematic progression from one to three solutions under increasingly restrictive hypotheses is absent from the literature.

Our contribution in this manuscript addresses these gaps by developing a comprehensive framework for higher-order nonlinear dynamic equations with mixed boundary conditions. In particular, we investigate the existence of solutions for a nonlinear boundary value problem with mixed boundary conditions, where nonlinear reaction functions both in the interior and on the boundary of the domain satisfy polynomial growth. Prior to introducing our problem, we recall some standard definitions and notations of time scale calculus that are relevant for our purposes. For further details on time scale calculus, we refer readers to \cite{BP}.

Let $\mathbb{T}$ be a time scale. Let $0, T \in \mathbb{T}$ with $T > 1$, and define  $J := [0, T]_{\mathbb{T}}$, which is nonempty by the fact $T>1$.

\begin{definition}
Let $t \in J.$ We define the forward jump operator $\sigma : J \rightarrow J$ by $$\sigma(t) := \inf\{s \in J : s > t\}.$$
We define the backward jump operator $\rho : J \rightarrow J$ by $$\rho(t) := \sup\{s \in J : s < t\}.$$
In this definition, we specify $\inf \emptyset = \sup J = T$ and $\sup \emptyset = \inf J = 0.$
\end{definition}
\begin{definition}
For the purpose of defining differentiation, we need to specify the time scale $$J^k = J - \{ T \}.$$
For the purpose of defining our boundary value problem, we need to specify the time scale $$J^0 = J - \{ 0,T \}.$$
\end{definition}

\begin{definition}
	Assume $f: J \rightarrow \mathbb{R}$ is a function and let $t \in J^k.$  Then we define $f^{\Delta}(t)$ to be the number (provided it exists) with the property that given any $\varepsilon > 0,$ there is a neighborhood $U$ of $t$ such that $$\left|[f(\sigma(t)) - f(s)] - f^{\Delta}(t)[\sigma(t) - s]\right| \leq \varepsilon |\sigma(t) - s|,$$ for all $s \in U.$  We call $f^{\Delta}(t)$ the delta (or Hilger) derivative of $f$ at $t.$  We also make note that $f^{\Delta \Delta}(t) = (f^{\Delta})^{\Delta}(t).$
\end{definition}
Next, we introduce the following nonlinear boundary value problem that is under consideration:
\begin{align}
\label{eq:main}
(-1)^n\Delta^nu(\rho(t)) + f(t, u(\rho(t)), \ldots, \Delta^{n-1}u(\rho(t))) &= 0, \quad t \in (0, T),\notag \\
\Delta^{n-1}u(0) &= g_{n-1}(u(0)), \notag\\
\Delta^{n-2}u(\sigma(0)) &= g_{n-2}(u(\sigma(0))), \notag\\
&\vdots \notag\\
\Delta u(\sigma^{n-2}(0)) &= g_1(u(\sigma^{n-2}(0))), \notag\\
u(\sigma^{n-1}(0)) &= u(T) = 0, 
\end{align}
where $n \in \mathbb{N}$, and $\sigma^{n-1}(0) \leq T$. Furthermore, we assume the following hypotheses.
\begin{description}
\item[(A1)]$f \in C(J \times \mathbb{R}^n)$ satisfies
\begin{align}
|f(t, x_1, x_2, \ldots, x_n)| \leq b_0(t) + \sum_{j=1}^{n} b_j(t)|x_j|^{k_j}, \quad t \in J, \label{eq:f-bound}
\end{align}
where $k_j \geq 0$ for $j \in \{1, \ldots, n\}$, and $b_l \in C(J)$ with $0 \leq b_l \leq B$ on $J$, for $l \in \{0, 1, \ldots, n\}$, and  $g_j \in C(\mathbb{R})$, $j \in \{1, \ldots, n-1\}$, satisfy
\begin{align}
|g_j(x)| \leq a_{0j} + \sum_{k=1}^{p_j} a_{kj}|x|^{l_k}, \quad j \in \{1, \ldots, n\}, \label{eq:g-bound}
\end{align}
where $0 \leq a_{0j}, a_{kj} \leq B$ with $B$ being a positive constant for $j \in \{1, \ldots, n-1\}$ and $l_k \geq 0$ for $k \in \{1, \ldots, p_j\}$. 
Define
\begin{align}
B_1 := \max\left(2B + nB^{n+1}, \max_{j\in\{1,\ldots,n\}}\left(2B + \sum_{k=1}^{p_j} B^{1+l_k}\right)\right). \label{eq:B1-def}
\end{align}
\item[(A2)] Let  $r$, $L$, $R$ be positive constants satisfying
\begin{align}
r < L < R. \label{eq:A2-cond}
\end{align}
\end{description}
\begin{definition}
We introduce the following function spaces for our purposes. Let $X_1 = C^n_{\rd}(J)$ be the space endowed with the norm
\begin{align}
\|u\|_1 = \max_{j\in\{0,1,\ldots,n\}} \max_{t\in J} |\Delta^j u(t)|, \notag
\end{align}
provided it exists. Then we define the product space $X := (X_1)^{n+2}$ with the norm
\begin{align}
\|v\| = \max_{j\in\{1,2,\ldots,n+2\}} \|v_j\|_1, \quad v = (v_1, \ldots, v_n, v_{n+1}, v_{n+2}). \notag
\end{align}    
\end{definition}
Note that, $\langle X_1,\|\cdot\|_1\rangle$ is a Banach space and so is $\langle X, \|\cdot \|\rangle $. 

Having introduced the necessary definitions and notations, we now state our main results on the existence of non-negative solutions to Problem \eqref{eq:main}.
\begin{theorem} \label{thm:one-sol}
Given Hypothesis {\bf(A1)},  Problem \eqref{eq:main} admits at least one classical solution in $X$.
\end{theorem}

\begin{theorem} \label{thm:three-sol}
Given Hypotheses {\bf(A1)} and {\bf(A2)}, Problem \eqref{eq:main} admits at least three nonnegative classical solutions $X$.
\end{theorem}

The remainder of this paper is organized as follows. In Section \ref{prelim}, we present the preliminary results necessary for our analysis, including fundamental definitions from functional analysis and essential fixed-point theorems for expansive mappings that form the theoretical foundation of our approach. Section \ref{one-solution} establishes the existence of at least one classical solution, not necessarily nonnegative, to Problem \eqref{eq:main}.  Section \ref{three-solutions} further develops our existence result to guarantee at least three nonnegative solutions by systematically verifying the assumptions of our fixed-point framework. Finally, Section \ref{example} provides a concrete illustrative example that validates our theoretical results. Our methodology progresses systematically from establishing basic existence to proving multiple solution results, with each section building upon the previous theoretical developments while maintaining the unified framework of fixed-point theory.

\section{Preliminary Results}
\label{prelim}

In this section, we begin by recalling some standard definitions from functional analysis, followed by a review of recently established results in fixed point theory that are relevant for our purposes.

\begin{definition}[{\bf Completely Continuous Operator}]
Let $K: M \subset X \to X$ be an operator, where $X$ is a real Banach space. We say that $K$ is compact if $K(M)$ is contained in a compact subset of $X$. $K$ is called a completely continuous operator if it is continuous and maps any bounded set into a relatively compact set.
\end{definition}
\begin{definition}[{\bf Expansive Operator}]
Let $X$ and $Y$ be real Banach spaces. An operator $K: X \to Y$ is called expansive if there exists a constant $h > 1$ such that
\begin{align}
\|Kx - Ky\|_Y \geq h\|x - y\|_X \notag
\end{align}
for any $x, y \in X$.
\end{definition}
\begin{definition}[{\bf Cone}]
A closed, convex set $P$ in $X$ is said to be a cone if:
\begin{enumerate}
\item $\alpha x \in P$ for any $\alpha \geq 0$ and for any $x \in P$,
\item $x, -x \in P$ implies $x = 0$.
\end{enumerate}
\end{definition}
Proofs of our existence results rely heavily on fixed point theory. Below we recall some fixed point theory available in existing literature. 
\begin{pro}
\label{thm:fixed-point1}
\cite[Thm. \ 2.1]{4}, \cite[Thm. \ 2.1]{7} Let $E$ be a Banach space, $Y$ a closed, convex subset of $E$, and $U$ an open subset of $Y$ with $0 \in U$. Consider two operators $T$ and $S$, where
\begin{align}
Tx = \eta x, \quad x \in U, \notag
\end{align}
for $\eta > 1$, and $S: U \to E$ is continuous and satisfies:
\begin{itemize}
\item[(i)] $(I - S)(U)$ resides in a compact subset of $Y$,
\item[(ii)] the set $\{x \in \partial U : x = \lambda(I - S)x\}$ is empty for any $\lambda \in \left(0, \frac{1}{\eta}\right)$.
\end{itemize}
Then, $T + S$ has at least one fixed point in $U$.
\end{pro}

\begin{pro} \label{thm:fixed-point3}
\cite[Thm. 2.3]{3}, \cite[Thm. 2.3]{15}
Let $P$ be a cone in a Banach space $E$, $\Omega$ a subset of $P$, and $U_1$, $U_2$, and $U_3$ three open bounded subsets of $P$ such that $U_1 \subset U_2 \subset U_3$ and $0 \in U_1$. Assume that $T: \Omega \to E$ is an expansive mapping, $S: U_3 \to E$ is completely continuous, and $S(U_3) \subset (I - T)(\Omega)$.

Suppose that $(U_2 \setminus U_1) \cap \Omega \neq \emptyset$, $(U_3 \setminus U_2) \cap \Omega \neq \emptyset$, and there exist $w_0 \in P^*$ and $\eta > 0$ small enough such that:
\begin{itemize}
\item[(i)] $Sx \neq (I - T)(\lambda x)$ for all $\lambda \geq 1 + \eta$, $x \in \partial U_1$, and $\lambda x \in \Omega$,
\item[(ii)] $Sx \neq (I - T)(x - \lambda w_0)$ for all $\lambda \geq 0$ and $x \in \partial U_2 \cap (\Omega + \lambda w_0)$,
\item[(iii)] $Sx \neq (I - T)(\lambda x)$ for all $\lambda \geq 1 + \eta$, $x \in \partial U_3$, and $\lambda x \in \Omega$.
\end{itemize}

Then $T + S$ has at least three non-zero fixed points $x_1, x_2, x_3 \in P$ such that
\begin{align}
x_1 \in U_1 \cap \Omega, \quad x_2 \in (U_2 \setminus U_1) \cap \Omega, \quad \text{and} \quad x_3 \in (U_3 \setminus U_2) \cap \Omega. \notag
\end{align}
\end{pro}

\section{Existence of at Least One Solution}
\label{one-solution}

In this section, we establish that problem \eqref{eq:main} has at least one classical solution by employing Proposition \ref{thm:fixed-point1}. Note that this solution can be nonnegative, nonpositive or sign changing. We begin by defining appropriate operators corresponding to our problem \eqref{eq:main} to facilitate our analysis. Then we carefully verify all the assumptions outlined in Proposition \ref{thm:fixed-point1}.

For any $u \in X$, we define the operators:
\begin{align}
S_{11}u(t) &= (-1)^n\Delta^n u_1(\rho(t)) + f\left(t, u_1(\rho(t)), \ldots, \Delta^{n-1}u_1(\rho(t))\right), \notag\\
S_{12}u(t) &= \Delta^{n-1}u_1(0) - g_{n-1}(u_1(0)), \notag\\
S_{13}u(t) &= \Delta^{n-2}u_1(\sigma(0)) - g_{n-2}(u_1(\sigma(0))), \notag\\
&\vdots \notag\\
S_{1n}u(t) &= \Delta u_1(\sigma^{n-2}(0)) - g_1(u_1(\sigma^{n-2}(0))), \notag\\
S_{1,n+1}u(t) &= u_1(\sigma^{n-1}(0)), \notag\\
S_{1,n+2}u(t) &= u_1(T), \notag
\end{align}
Then,
\begin{equation}
\label{eq:operator}
 S_1u(t) = (S_{11}u(t), S_{12}u(t), \ldots, S_{1,n+2}u(t)) \mbox{ for }u(t) = (u_1(t), \ldots, u_{n+2}(t)), t \in J   
\end{equation}
is the operator corresponding to \eqref{eq:main}. Note that, if $u = (u_1, \ldots, u_{n+2})$ is a solution to the equation $S_1u(t) = 0$ for $t \in J$, then $u$ solves problem \eqref{eq:main}. 

In what follows, we prove a series of lemmas that facilitate the verification of all the assumptions of Proposition \ref{thm:fixed-point1} for the operator \ref{eq:operator}. 

\begin{lemma} \label{lem:S1-bound}
Given Hypothesis {\bf (A1)}, if $u \in X$ with $\|u\| \leq B$, then
\begin{align}
|S_{1j}u(t)| \leq B_1, \quad j \in \{1, \ldots, n+2\}, \quad t \in J. \notag
\end{align}
\end{lemma}

\begin{proof}
From \eqref{eq:f-bound} and \eqref{eq:g-bound}, for $ t \in J$, we have
\begin{align}
\left|g_j\left(u\left(\sigma^{n-j-1}(0)\right)\right)\right| &\leq a_{0j} + \sum_{k=1}^{p_j} a_{kj}\left|u\left(\sigma^{n-j-1}(0)\right)\right|^{l_k} \notag
\leq B + \sum_{k=1}^{p_j} B^{1+l_k}, \quad j \in \{1, \ldots, n-1\}, \notag
\end{align}
and
\begin{align}
|f(t, u(\rho(t)), \ldots, \Delta^{n-1}u(\rho(t)))| &\leq b_0(t) + \sum_{j=1}^{n} b_j(t)|\Delta^{j-1}u(\rho(t))|^{k_j} \notag \leq B + \sum_{j=1}^{n} B^{n+1} = B + nB^{n+1}. \notag
\end{align}
Then it follows, for $j=1$,
\begin{align}
|S_{11}u(t)| &= |(-1)^n\Delta^n u_1(\rho(t)) + f(t, u_1(\rho(t)), \ldots, \Delta^{n-1}u_1(\rho(t)))| \notag\\
&\leq |\Delta^n u_1(\rho(t))| + |f(t, u_1(\rho(t)), \ldots, \Delta^{n-1}u_1(\rho(t)))| \notag\\
&\leq B + (B + nB^{n+1}) = 2B + nB^{n+1} \leq B_1, \notag
\end{align}
and for $j \in \{2, \ldots, n\}$, 
\begin{align}
|S_{1j}u(t)| &= |\Delta^{n-j+1}u_1(t) - g_{n-j+1}(u_1(0))| \notag\\
&\leq |\Delta^{n-j+1}u_1(0)| + |g_{n-j+1}(u_1(0))| \notag\\
&\leq B + \left(B + \sum_{k=1}^{p_j} B^{1+l_k}\right) = 2B + \sum_{k=1}^{p_j} B^{1+l_k} \leq B_1. \notag
\end{align}
Note that, for $j=n+1$ and $j=n+2$, using the fact that $\|u\| \le B$, we readily obtain,
\begin{align}
|S_{1,n+1}u(t)| &= |u_1(\sigma^{n-1}(0))| \leq B \leq B_1, \notag\\
|S_{1,n+2}u(t)| &= |u_1(T)| \leq B \leq B_1, \notag
\end{align}
Therefore, we can conclude that, $|S_{1j}u(t)| \leq B_1$ for all $j \in \{1, \ldots, n+2\}$ and $t \in J$.
\end{proof}

We now introduce the following operator, which plays a crucial role in the subsequent lemmas:
\begin{align}
S_2(u)(t) = \frac{A}{T^{n+1}} \int_0^t h_n(t, \sigma(s))S_1u(s)ds, \label{eq:S2-def}
\end{align}
where $u \in X$ and $A$ is a positive constant.

\begin{lemma} \label{lem:S2-bound}
Given Hypothesis {\bf (A1)}, if $u \in X$ with $\|u\| \leq B$, then
\begin{align}
\|S_2u\| \leq AB_1. \notag
\end{align}
\end{lemma}

\begin{proof}
Observe that, for any $k \in \{0, 1, \ldots, n\}$, 
\begin{align}
\Delta^k S_2u(t) = \frac{A}{T^{n+1}} \int_0^t h_{n-k}(t, \sigma(s))S_1u(s)\Delta s, \quad t \in J. \notag
\end{align}
Now, using Lemma \ref{lem:S1-bound}, we obtain
\begin{align}
|\Delta^k S_{2j}u(t)| &= \left|\frac{A}{T^{n+1}} \int_0^t h_{n-k}(t, \sigma(s))S_{1j}u(s)\Delta s\right| \notag\\
&\leq \frac{A}{T^{n+1}} \int_0^t h_{n-k}(t, \sigma(s))|S_{1j}u(s)|\Delta s \notag\\
&\leq \frac{A}{T^{n+1}} \cdot T^{n-k+1} \cdot B_1 \leq \frac{A}{T^{n+1}} \cdot T^{n+1} \cdot B_1 = AB_1 \notag
\end{align}
for all $t \in J$ and $j \in \{1, \ldots, n+1\}$ which implies $\|S_2u\| \leq AB_1$.
\end{proof}

\begin{lemma} \label{lem:S2-sol}
Given Hypothesis {\bf (A1)}, if $u \in X$ satisfies
\begin{align}
S_2(u)(t) = C, \quad t \in J, \label{eq:S2-const}
\end{align}
for some constant $C$, then $u$ is a solution to problem \eqref{eq:main}.
\end{lemma}

\begin{proof}
Consider taking $(n+1)^{\text{st}}$ derivative of equation \eqref{eq:S2-const} with respect to $t$. This immediately yields
\begin{align}
\frac{A}{T^{n+1}}S_1(u)(t) = 0, \quad t \in J, \notag
\end{align}
which implies $S_1(u)(t) = 0$ for $t \in J$. Therefore, $u$ is a solution to the problem \eqref{eq:main}.
\end{proof}
We now proceed to prove Theorem \ref{thm:one-sol}.
Let $Y_e$ denote the set of all equi-continuous families in $X$ with respect to the norm $\|\cdot\|$. Set $Y=Y_e$ and define
\begin{align}
U = \left\{u \in Y : \|u\| < B \text{ and if } \|u\| \geq \frac{B}{2}, \text{ then } u(0) > \frac{B}{2}\right\}. \notag
\end{align}

For $u \in U$, $\eta > 1$, and $t \in J$, define operators $T$ and $S$ as follows:
\begin{align}
Tu(t) &= \eta u(t), \notag \quad
Su(t) = u(t) - \eta u(t) - \eta S_2u(t). \notag
\end{align}

Now using Lemma \ref{lem:S2-bound}, we obtain, for all $u \in U$, 
\begin{align}
\|(I - S)u\| = \|\eta u + \eta S_2u\| \leq \eta\|u\| + \eta\|S_2u\| \leq \eta B + \eta AB_1. \notag
\end{align}
Thus, $S: U \to X$ is rd-continuous and $(I - S)(U)$ resides in a compact subset of $Y$.

Next, suppose there exists $u \in \partial U$ such that $u = \lambda(I - S)u$ or equivalently, $u = \lambda\eta(u + S_2u)$, for some $\lambda \in \left(0, \frac{1}{\eta}\right)$. Since $S_2u(0) = 0$ and $\|u\| = B > \frac{B}{2}$, we have $u(0) > \frac{B}{2}$, which gives $u(0) = \lambda\eta u(0)$. This implies $\lambda\eta = 1$, contradicting our assumption that $\lambda < \frac{1}{\eta}$. Therefore,
\begin{align}
\{u \in \partial U : u = \lambda_1(I - S)u\} = \emptyset \notag
\end{align}
for any $\lambda_1 \in \left(0, \frac{1}{\eta}\right)$ and by Proposition \ref{thm:fixed-point1}, the operator $T + S$ has a fixed point $u^* \in Y$ and $S_2u^*(t) = 0$ for $t \in J$. Indeed,
\begin{align}
u^*(t) &= Tu^*(t) + Su^*(t) \notag = \eta u^*(t) + u^*(t) - \eta u^*(t) - \eta S_2u^*(t) \notag = u^*(t) - \eta S_2u^*(t). \notag
\end{align}
Hence by Lemma \ref{lem:S2-sol}, $u^*$ is a solution to Problem \eqref{eq:main}. This completes the proof.

\section{Existence of at Least Three Nonnegative Solutions}
\label{three-solutions}
In this section, we establish Theorem \ref{thm:three-sol} by employing Proposition \ref{thm:fixed-point3}. This theorem significantly improves Theorem \ref{thm:one-sol}, as this theorem guarantees the existence of nonnegative solutions while simultaneously guarantee the non uniqueness in other other words multiplicity of solutions. Our approach involves defining the cone 
$P$, introducing the necessary operators, selecting suitable open bounded subsets within 
$P$, and systematically verifying each of the assumptions outlined in Proposition \ref{thm:fixed-point3}.

Let $P_e = \{u \in X : u \geq 0 \text{ on } J \}$ and $P$ denote the set of all equi-continuous families in $P_e$. For $v \in X$ and $m>0$, define operators
\begin{align}
\label{op:T1S3}
T_1v(t) &= (1 + m)v(t), \quad
S_3v(t) = - |S_2(v)(t)| - m v(t), 
\end{align}
where $t \in J$ and $S_2$ is defined by \eqref{eq:S2-def}. Observe that, since $m>0$, $T_1: P \to X$ is an expansive mapping. Moreover, any fixed point $v \in X$ of $T_1 + S_3$ is a solution to Problem \eqref{eq:main}.

Now define the sets
\begin{align}
\label{sets}
\Omega &:= P, \notag \\ 
U_1 &:= P_r = \{v \in P : \|v\| < r\}, \notag\\
U_2 &:= P_L = \{v \in P : \|v\| < L\}, \notag \\
U_3 &:= P_{R} = \{v \in P : \|v\| < R\}. 
\end{align}

Note that, since it is not required for \( \Omega \) to be a strict subset of \( P \) in order to apply Proposition~\ref{thm:one-sol}, it is sufficient to take \( \Omega \) to be the entire set \( P \). By the construction of the sets $U_1,~U_2,~U_3$, it is evident from Hypothesis {\bf (A2)} that $U_1 \subset U_2 \subset U_3$, $0 \in U_1$ and $(U_2 \setminus U_1) \cap \Omega \neq \emptyset$, $(U_3 \setminus U_2) \cap \Omega \neq \emptyset$, which follows from the fact $\Omega = P$. 
\begin{rem}{\rm
Observe that, $S_3(U_3)$ forms a bounded set of equicontinuous functions in the function space $X$. By the Arzela-Ascoli theorem, a subset of continuous functions is relatively compact if and only if it is both bounded and equicontinuous. Since $S_3(U_3)$ satisfies both conditions, it is relatively compact, which implies that the operator $S_3: U_3 \to X$ is completely continuous (compact).}
\end{rem}

\begin{rem}{\rm Note that, the relation $(I-T_1)(\Omega) = -(1+m)\Omega = -\Omega$ shows that applying the operator $(I-T_1)$ to the set $\Omega$ yields $-\Omega$. The condition $S_3 u \leq 0$ for any $u \in U_3$ ensures that the operator $S_3$ maps every element of $U_3$ to non-positive values. Since $S_3 u \leq 0$ and we are working in an ordered space, we have $S_3 u \in (-\Omega)$ for any $u \in U_3$. Combining the above, we obtain $S_3(U_3) \subset (-\Omega) = (I-T_1)(\Omega)$, establishing the desired inclusion. This type of inclusion argument is fundamental in degree theory and fixed point theorems, where demonstrating that one operator's range is contained within another operator's image is essential for establishing topological properties.}
\end{rem}

Now, let $\eta=\frac{2AB_1}{mr}$. Clearly, \( \eta > 0 \), and \( \eta \) can be made sufficiently small by choosing \( m \) large enough.
We now verify Assumptions (i) and (iii) outlined in Proposition \ref{thm:fixed-point3} in the subsequent two steps.

\subsection{Step 1:} { \it To show, $S_3x \neq (I - T_1)(\lambda x)$ for all $\lambda \geq 1 + \eta$, $x \in \partial U_1$, and $\lambda x \in \Omega$}. On the contrary, let us assume that there exist $\lambda_1 \geq 1 + \eta$, $u \in \partial U_1$, and $\lambda_1 u \in \Omega$ such that $S_3(u) = (I - T_1)(\lambda_1 u)$. Then $
- |S_2(u)| - m u = -m\lambda_1 u $ which implies $|S_2(u)|+m u = m\lambda_1 u.$ This gives us
\begin{align}
m\lambda_1 r=m\lambda_1\|u\|\leq \|S_2(u)\|+m\|u\|\leq AB_1+mr,\notag
\end{align}
whereupon $\lambda_1\leq 1+\frac{AB_1}{mr}$.
This  is a contradiction.  Thus, assumption (i) of Proposition \ref{thm:fixed-point3} is satisfied.

\subsection{Step 2:}{\it To show, $S_3x \neq (I - T_1)(\lambda x)$ for all $\lambda \geq 1 + \eta$, $x \in \partial U_3$, and $\lambda x \in \Omega$.} We once again use a proof by contradiction, as in \textbf{Step 1}.
Note that, if there exist $\lambda_1 \geq 1 + \eta$, $u \in \partial U_3$, and $\lambda_1 u \in \Omega$ such that $S_3(u) = (I - T_1)(\lambda_1 u)$, we readily obtain
\begin{align}
m\lambda_1 R=m\lambda_1\|u\|\leq \|S_2(u)\|+m\|u\|\leq AB_1+mR,\notag
\end{align}
whereupon $\lambda_1\leq 1+\frac{AB_1}{mR}\leq 1+\frac{AB_1}{mr}$,
which is again a contradiction. Therefore, assumption (iii) of Proposition \ref{thm:fixed-point3} is satisfied.

\vspace{0.1 in}

\noindent Finally, we verify assumption (ii) of Proposition \ref{thm:fixed-point3}. For that, let $u_0 \in P^*$, where $P^*$ is the dual space of $P$.
\subsection{Step 3:} {\it To show, $S_3x \neq (I - T_1)(x - \lambda u_0)$ for all $\lambda \geq 0$ and $x \in \partial U_2 \cap (\Omega + \lambda w_0)$.} On the contrary, let us assume there exist $\lambda_1 \geq 0$ and $u \in \partial P_L \cap (\Omega + \lambda_1 u_0)$ such that $S_3(u) = (I - T_1)(u - \lambda_1 u_0)$. Then
$- |S_2(u)| - m u  = -m(u - \lambda_1 u_0)$ which implies $-|S_2(u)| = \lambda_1 mu_0 $, which is a contradiction.

\vspace{0.1 in}

\noindent Consequently, by Proposition \ref{thm:fixed-point3}, Problem \eqref{eq:main} has at least three nonnegative solutions $u_1$, $u_2$, and $u_3$ satisfying either
\begin{align}
u_1 \in \partial U_1 \cap \Omega, \quad u_2 \in (U_2 \setminus U_1) \cap \Omega, \quad \text{and} \quad u_3 \in (U_3 \setminus U_2) \cap \Omega, \notag
\end{align}
or
\begin{align}
u_1 \in U_1 \cap \Omega, \quad u_2 \in (U_2 \setminus U_1) \cap \Omega, \quad \text{and} \quad u_3 \in (U_3 \setminus U_2) \cap \Omega. \notag
\end{align}

\section{An Example}
\label{example}

We conclude this paper by presenting an example to illustrate our theoretical results. Let $\mathbb{T} = 4\mathbb{N}_0 \cup \{0\}$, $T = 256$, $n = 2$, and define
\begin{align}
a_{01}(t) &= \frac{1}{2}, \quad a_{11}(t) = \frac{1}{3}, \quad a_{21}(t) = \frac{1}{4}, \notag\\
b_0(t) &= 1, \quad b_1(t) = \frac{1}{100}, \quad b_2(t) = \frac{1}{10}, \quad t \in [0, 256], \notag
\end{align}
with $k_1 = k_2 = l_1 = 1$, $p_1 = l_2 = 2$, and
\begin{align}
f(t, x, y) &= 1 + \frac{x}{100(1 + x^2)} + \frac{y}{10(1 + y^2 + y^4)}, \quad t \in [0, 256], \quad x, y \in \mathbb{R}, \notag\\
g_1(x) &= \frac{1}{2} + \frac{x}{3(1 + x^2 + x^4)} + \frac{x^2}{4(1 + x^2)}, \quad x \in \mathbb{R}. \notag
\end{align}

We choose the parameters $R = 10$, $L = 5$, $r = 4$, $m = 1050$, $B = 1$, $A = \frac{1}{10B_1}$. With $B_1 = 4$, we verify that
\begin{align}
r<L<R.
\end{align}

Therefore, assumptions {\bf (A1)} and  {\bf (A2)} are satisfied for the problem
\begin{align}
\label{eq:example}
\Delta^2 u\left(\frac{t}{4}\right) + 1 + \frac{u\left(\frac{t}{4}\right)}{100\left(1 + u\left(\frac{t}{4}\right)^2\right)} + \frac{\Delta u\left(\frac{t}{4}\right)}{10\left(1 + \Delta u\left(\frac{t}{4}\right)^2 + \left(\Delta u\left(\frac{t}{4}\right)\right)^4\right)} &= 0, \quad t \in [4, 256), \notag\\
\Delta u(0) = \frac{1}{2} + \frac{u(0)}{3(1 + (u(0))^2 + (u(0))^4)} + \frac{(u(0))^2}{4(1 + (u(0))^2)}&, \notag\\
u(1) = u(256) &= 0. 
\end{align}
and by applying Theorems \ref{thm:one-sol} and \ref{thm:three-sol}, one can conclude that Problem \eqref{eq:example} has at least one solution and at least three nonnegative solutions, respectively.

\end{document}